\documentclass[a4paper,10pt,leqno]{amsart}
\usepackage[cp850]{inputenc}

\usepackage[dvips]{graphics}

\usepackage{psfrag}

\usepackage{amsxtra,amssymb}

\usepackage{mathrsfs}

\numberwithin{equation}{section}

\newtheorem{theo}{Theorem}[section]

\newtheorem{prop}[theo]{Proposition}

\newtheorem{lemm}[theo]{Lemma}

\newtheorem{coro}[theo]{Corollary}

\theoremstyle{definition}

\theoremstyle{remark}

\begin{document}

\baselineskip=10pt

\title{Two estimates concerning classical diophantine approximation constants}

\author{ Johannes Schleischitz}
                            
\maketitle

\vspace{8mm}

\begin{abstract}
In this paper we aim to prove two inequalities involving the classical approximation constants $w_{n}^{\prime}(\zeta),\widehat{w}_{n}^{\prime}(\zeta)$
that stem from the simultaneous approximation problem $\vert \zeta^{j}x-y_{j}\vert$, $1\leq j\leq n$, on the one side and the constants 
$w_{n}^{\ast}(\zeta),\widehat{w}_{n}^{\ast}(\zeta)$ connected to approximation with algebraic numbers of degree $\leq n$ on the other side. We concretely prove 
$w_{n}^{\ast}(\zeta)\widehat{w}_{n}^{\prime}(\zeta)\geq 1$ and $\widehat{w}_{n}^{\ast}(\zeta)w_{n}^{\prime}(\zeta)\geq 1$.
The first result is due to W. Schmidt, however our method of proving it allows to derive the other inequality as a dual result.\\
Finally we will discuss estimates of $w_{n}^{\ast}(\zeta), \widehat{w}_{n}^{\ast}(\zeta)$ uniformly in $\zeta$ depending only on $n$ as an application.

\end{abstract}
 \vspace{4mm}
 \indent {\footnotesize{ Supported by FWF grant P22794-N13}}

\vspace{8mm}

\section{Introduction}

\maketitle

\subsection{Approximation constants $w_{n}^{\prime}(\zeta),\widehat{w}_{n}^{\prime}(\zeta)$ and $w_{n}^{\ast}(\zeta),\widehat{w}_{n}^{\ast}(\zeta)$} 

For a fixed positive integer $n$ and a vector $\boldsymbol{\zeta}=(\zeta_{1},\zeta_{2},\ldots ,\zeta_{n})\in{\mathbb{R}^{n}}$ define the approximation constants
$w_{n,j}^{\prime}(\zeta), 1\leq j\leq n+1$, as the supremum of all real numbers $\nu$, such that the system 

\begin{equation}  \label{eq:1}
 \vert x\vert \leq X, \quad \vert \zeta_{i}x- y_{i}\vert \leq X^{-\nu}, \qquad 1\leq i \leq n,
\end{equation}

\noindent has $j$ linearly independent solution $(x,y_{1},\ldots ,y_{n})$ for certain arbitrary large values of $X$. Similarly, define
 $\widehat{w}_{n,j}^{\prime}(\zeta), 1\leq j\leq n+1$ as the supremum of all $\nu$, such that system (\ref{eq:1}) has $j$ linearly
independent solutions for all sufficiently large $X$. Clearly $\widehat{w}_{n,j}^{\prime}(\zeta)\leq w_{n,j}^{\prime}(\zeta)$ for all $1\leq j\leq n+1$.
In the special case $\boldsymbol{\zeta}=(\zeta,\zeta^{2},\ldots ,\zeta^{n})$ for some real number $\zeta$, which will be in the focus of our study,
let $w_{n,j}^{\prime}(\zeta):=w_{n,j}^{\prime}(\boldsymbol{\zeta})$ and similarly $\widehat{w}_{n,j}^{\prime}(\zeta):=\widehat{w}_{n,j}(\boldsymbol{\zeta})$.
For convenient writing we further put $w_{n}^{\prime}(\boldsymbol{\zeta}):= w_{n,1}^{\prime}(\boldsymbol{\zeta}), 
\widehat{w}_{n}^{\prime}(\boldsymbol{\zeta}):=\widehat{w}_{n,1}^{\prime}(\boldsymbol{\zeta})$. In particular $w_{n}^{\prime}(\zeta)$ (resp. $\widehat{w}_{n}^{\prime}(\zeta)$) 
is the supremum of all $\nu$ such that (\ref{eq:1}) with $\zeta_{i}=\zeta^{i}$ has infinitely many solutions respectively a solution for all sufficiently large $X$.     \\

\noindent The approximation constants $w_{n}^{\ast}(\zeta),\widehat{w}_{n}^{\ast}(\zeta)$ quantify how good a 
real number can be approximated by algebraic numbers of degree at most $n$.
For a polynomial $P\in{\mathbb{Z}[T]}$ define $H(P)$ as the largest absolute value of its coefficients and for an algebraic number $\alpha$ define
 $H(\alpha):=H(P)$ for the minimal polynomial $P$ of $\alpha$ with relatively prime integral coefficients. The constants $w_{n}^{\ast}(\zeta)$ 
(resp. $\widehat{w}_{n}^{\ast}(\zeta)$) are given by the supremum of all real $\nu$ such that 

\begin{equation} \label{eq:2}
 \vert \zeta-\alpha\vert \leq H(\alpha)^{-\nu-1}
\end{equation}

\noindent has infintely many solutions, respectively a solution for arbitrarily large values of $H(\alpha)$, with $\alpha$ algebraic of degree $\leq n$. 
Clearly $w_{n}^{\ast}(\zeta), \widehat{w}_{n}^{\ast}(\zeta)$ are monotonically increasing as $n$ increases. For $w_{n}^{\ast}(\zeta)$ it is known that we have 

\begin{equation}  \label{eq:anelim}
\frac{n+1}{2}\leq w_{n}^{\ast}(\zeta)
\end{equation}

\noindent for all $\zeta\in{\mathbb{R}}$ not algebraic of degree $\leq n$ as a consequence of (\ref{eq:hundk}) and (\ref{eq:mily})                                                                                                 
we will establish later. \\
It is however conjectured that even the stronger  
lower bound $w_{n}^{\ast}(\zeta)\geq n$ holds for any $\zeta$ not algebraic of degree $\leq n$.      
It is well known that $n$ is the optimal possible
uniform (in $\zeta$) lower bound in (\ref{eq:anelim}), as for a generic $\zeta$ we have $w_{n}^{\ast}(\zeta)=\widehat{w}_{n}^{\ast}(\zeta)=n$,
 see Theorem 2.3 in \cite{1} for example.                                                                                  

\subsection{Related approximation problems and constants}

We will only treat the case $\boldsymbol{\zeta}=(\zeta,\zeta^{2},\ldots ,\zeta^{n})$ in the sequel which is sufficient for our concern,
although most of the following results in the introduction hold for any vector $\boldsymbol{\zeta}\in{\mathbb{R}^{n}}$.\\
In order to study the values $w_{n}^{\prime}(\zeta),\widehat{w}_{n}^{\prime}(\zeta)$ we first introduce a closely connected Diophantine 
approximation problem. Consider the system 

\begin{eqnarray}
 \vert x\vert &\leq& Q^{1+\theta} \label{eq:3}  \\
 \vert \zeta x-y_{1}\vert &\leq& Q^{-\frac{1}{n}+\theta}  \nonumber  \\
 \vert \zeta^{2}x-y_{2}\vert &\leq& Q^{-\frac{1}{n}+\theta}   \nonumber    \\
 \vdots \quad &\vdots& \quad \vdots   \nonumber \\
 \vert \zeta^{n}x-y_{n}\vert &\leq& Q^{-\frac{1}{n}+\theta},   \nonumber
\end{eqnarray}

\noindent parametrised by $Q>1$. Let $\psi_{n,j}(Q)$ be the infimum of all $\theta$ such that (\ref{eq:3}) has $j$ linearly independent solutions.
The functions $\psi_{n,j}(Q)$ can alternatively be interpreted via successive minima. Consider the lattice
$\Lambda=\{(x,\zeta x-y_{1},\ldots ,\zeta^{n}x-y_{n}):x,y_{1},\ldots y_{n}\in{\mathbb{Z}} \}$
and the convex body (in fact the parallelepiped) $K(Q)$ defined as the set of points 
$(z_{1},z_{2},\ldots ,z_{n+1})\in{\mathbb{R}^{n+1}}$ with

\begin{eqnarray}
  \vert z_{1}\vert &\leq& Q    \label{eq:tommot}  \\
  \vert z_{i}\vert &\leq& Q^{-\frac{1}{n}},\qquad 2\leq i\leq n+1.   \label{eq:tomot}
\end{eqnarray}

\noindent Now define $\lambda_{n,j}(Q)$ to be the $j$-th {\em successive minimum} of $K(Q)$ with respect to $\Lambda$, which by definiton
is the smallest value $\lambda$ such that $\lambda \cdot K(Q)$ contains (at least) $j$ linearly independent lattice points.
Then we have

\[
 Q^{\psi_{n,j}(Q)}=\lambda_{n,j}(Q).
\]

\noindent Put $\underline{\psi}_{n,j}:=\liminf_{Q\to\infty} \psi_{n,j}(Q), \overline{\psi}_{n,j}:= \limsup_{Q\to\infty} \psi_{n,j}(Q)$. 
We have $-1\leq \psi_{n,j}(Q)\leq \frac{1}{n}$ for all $Q>1$ and $1\leq j\leq n+1$ (which is implicitely derived rigurously from section
4 in \cite{6} considering the functions $L_{j}$ arising from $\psi_{n,j}$) and consequently                                                                          

\begin{eqnarray*}
 -1&\leq& \underline{\psi}_{n,1}\leq \underline{\psi}_{n,2}\leq \cdots \leq \underline{\psi}_{n,n+1}\leq \frac{1}{n}  \\
 -1&\leq& \overline{\psi}_{n,1}\leq \overline{\psi}_{n,2}\leq \cdots \leq \overline{\psi}_{n,n+1}\leq \frac{1}{n}.
\end{eqnarray*}

\noindent A crucial observation for the study of the functions $\psi_{n,j}(Q)$ (a special case of Theorem 1.1 in \cite{6})                                             
is that for $\zeta$ not algebraic of degree $\leq n$ and every $1\leq j\leq n$ there are infinitely many $Q$ 
with $\psi_{n,j}(Q)=\psi_{n,j+1}(Q)$. Thus in particular

\begin{equation}  \label{eq:stoiko}
 \underline{\psi}_{n,j+1}\leq \overline{\psi}_{n,j}, \qquad 1\leq j\leq n.
\end{equation}

\noindent An easy generalisation of Theorem 1.4 in \cite{6}                                                                                          
states that we have

\begin{equation} \label{eq:4}
 \left(1+\omega_{n,j}^{\prime}(\zeta)\right)\left(1+\underline{\psi}_{n,j}\right)= \left(1+\widehat{\omega}_{n,j}^{\prime}(\zeta)\right)\left(1+\overline{\psi}_{n,j}\right)= \frac{n+1}{n},\qquad 1\leq j\leq n+1,
\end{equation}

\noindent which enables us to easily compute the value $\underline{\psi}_{n,j}$ from $w_{n,j}^{\prime}(\zeta)$ such as $\overline{\psi}_{j}$ from 
$\widehat{w}_{n,j}^{\prime}(\zeta)$ and vice versa. \\

\noindent The approximation constants $w_{n}^{\ast}(\zeta), \widehat{w}_{n}^{\ast}(\zeta)$ are in close connection to the dual problem
of approximation of a linear form, which in our case is just $x+\zeta y_{1}+\ldots +\zeta^{n}y_{n}$. \\
First, define the approximation constants $w_{n,j}(\zeta)$ (resp. $\widehat{w}_{n,j}(\zeta)$) as the supremum of all $\nu\in{\mathbb{R}}$, such that the system 

\[
 \vert y_{i}\vert \leq X, \qquad \vert x+\zeta y_{1}+\ldots +\zeta^{n}y_{n} \vert \leq X^{-\nu}
\]

\noindent has $j$ linearly independent solutions $(x,y_{1},\ldots ,y_{n})$ for certain arbitrary large values of $X$ (respectively all sufficiently
large $X$) and put $w_{n}(\zeta):=w_{n,1}(\zeta),\widehat{w}_{n}(\zeta):=\widehat{w}_{n,1}(\zeta)$. In virtue of Dirichlet's Theorem we have

\begin{equation}  \label{eq:hundk}
 w_{n}(\zeta)\geq \widehat{w}_{n}(\zeta)\geq n.
\end{equation}
 
\noindent Given the dual lattice $\Lambda^{\ast}$ and the convex bodies $K^{\ast}(Q)$ dual to $K(Q)$, namely

\begin{eqnarray}
\Lambda^{\ast}&:=&\left\{(x+\zeta_{1}y_{1}+\zeta_{2}y_{2}+\cdots +\zeta_{n}y_{n},y_{1},y_{2},\ldots ,y_{n}):x,y_{1},\ldots ,y_{n}\in{\mathbb{Z}}\right\} \label{eq:lamd}   \\
K^{\ast}(Q)&:=&\left\{\boldsymbol{x}=(x,y_{1},\ldots ,y_{n}): \quad <\boldsymbol{x},\boldsymbol{z}>\leq 1 \quad \forall \boldsymbol{z}\in{K(Q)}\right\}  \label{eq:nonumbers}
\end{eqnarray}

\noindent we can define the functions $\lambda_{n,j}^{\ast}(Q)$ as the successive minima of $K^{\ast}(Q)$ with respect to $\Lambda^{\ast}$. 
Put

\[
 Q^{\psi_{n,j}^{\ast}(Q)}= \lambda_{n,j}^{\ast}(Q).
\]

\noindent Furthermore we denote by $\eta_{n,j}^{\ast}(Q)$ the successive minima of $K(Q)$ (instead of $K^{\ast}(Q)$) with respect to $\Lambda^{\ast}$, which
correspond to the successive minima of the convex body $K^{+}(Q)$ with respect to the lattice $\Lambda^{+}$ given by

\begin{eqnarray}
 \Lambda^{+}&:=& \mathbb{Z}^{n+1}     \label{eq:ask}         \\    
 K^{+} (Q) &:=& \left\{\boldsymbol{x}\in{\mathbb{R}^{n+1}}: \vert y_{t}\vert \leq Q^{\frac{1}{n}}\quad 1\leq t\leq n, \vert x+\zeta y_{1}+\cdots +\zeta^{n}y_{n}\vert\leq Q^{-1} \right\}    \label{eq:askk}
\end{eqnarray}

\noindent and define functions $\nu_{n,j}^{\ast}(Q)$ by $Q^{\nu_{n,j}^{\ast}(Q)}=\eta_{n,j}^{\ast}(Q)$.
If we put $\underline{\psi}_{n,j}^{\ast}:=\liminf_{Q\to\infty} \psi_{n,j}^{\ast}(Q), \overline{\psi}_{n,j}^{\ast}:=\limsup_{Q\to\infty} \psi_{n,j}^{\ast}(Q)$ 
and similarly define $\underline{\nu}_{n,j}^{\ast},\overline{\nu}_{n,j}^{\ast}$ the obvious inequalities $K^{\ast}(Q)\subset K(Q)\subset (n+1)K^{\ast}(Q)$ yield

\begin{equation}  \label{eq:mueli}
 \lim_{Q\to\infty} \psi_{n,j}^{\ast}(Q)-\nu_{n,j}^{\ast}(Q)=0 
\end{equation}

\noindent and thus

\begin{equation}   \label{eq:ka}
 \underline{\psi}_{n,j}^{\ast}= \underline{\nu}_{n,j}^{\ast}, \qquad \overline{\psi}_{n,j}^{\ast}= \overline{\nu}_{n,j}^{\ast}.
\end{equation}

\noindent Combined with Mahler's inequality $\lambda_{n,j}\lambda_{n,n+2-j}^{\ast}\asymp 1$ for $1\leq j\leq n+1$ we have

\begin{equation}  \label{eq:ka2}
 \underline{\psi}_{n,j}= -\overline{\psi}_{n,n+2-j}^{\ast}=-\overline{\nu}_{n,n+2-j}^{\ast}, \qquad \overline{\psi}_{n,j}= -\underline{\psi}_{n,n+2-j}^{\ast}=-\underline{\nu}_{n,n+2-j}^{\ast}.
\end{equation}

\noindent It will be more convenient to work with the functions $\nu_{n,j}^{\ast}$ in the sequel.
Again by generalizing Theorem 1.4 in \cite{6}                                                                                                               
we have 

\begin{equation}  \label{eq:ka3}
 \left(w_{n,j}(\zeta)+1\right)\left(\frac{1}{n}+\underline{\psi}_{n,j}^{\ast}\right)= \left(\widehat{w}_{n,j}(\zeta)+1\right)\left(\frac{1}{n}+\overline{\psi}_{n,j}^{\ast}\right)= \frac{n+1}{n}
\end{equation}

\noindent for $1\leq j\leq n+1$, and obviously by (\ref{eq:ka}) we can write $\underline{\nu}_{j}^{\ast}$ resp. $\overline{\nu}_{j}^{\ast}$ 
instead of $\underline{\psi}_{j}^{\ast}$ resp. $\overline{\psi}_{j}^{\ast}$. \\

\noindent A main ingredient in the proof of the Theorem \ref{thm2} will be Minkowski's Theorem, which asserts that
for a convex body $K\in{\mathbb{R}^{n+1}}$, a lattice $\Lambda$ and corresponding successive minima $\lambda_{n,j}$ we have

\begin{equation}  \label{eq:mink}
  \frac{2^{n+1}}{(n+1)!}\frac{\det(\Lambda)}{\rm{vol}(K)}\leq \lambda_{n,1}\lambda_{n,2}\cdots \lambda_{n,n+1}\leq 2^{n+1}\frac{\det(\Lambda)}{\rm{vol}(K)}.       
\end{equation}

\noindent For a proof see Theorem 1 page 60 and Theorem 2 page 62 in \cite{3}.\\                                                                                     
In our special case the height of the parallelepipeds $K^{+}(Q)$ in direction of the $x$-axis in every point $(x,y_{1},\ldots,y_{n})$ 
with $(y_{1},y_{2},\ldots ,y_{n})\in{\left[-Q^{\frac{1}{n}},Q^{\frac{1}{n}}\right]^{n}}$ is $2Q^{-1}$, so we have $\rm{vol}(K^{+}(Q))=\left(2Q^{\frac{1}{n}}\right)^{n}\cdot 2Q^{-1}=2^{n+1}$
 for all $Q>1$, and furthermore $\det(\Lambda^{+})=\det(\mathbb{Z}^{n+1})=1$. Since $\eta_{n,j}^{\ast}$ are the successive minima of $K^{+}(Q)$
with respect to $\Lambda^{+}$, (\ref{eq:mink}) leads to 

\[
 \frac{1}{(n+1)!}\leq \eta_{n,1}^{\ast}(Q)\eta_{n,2}^{\ast}(Q)\cdots \eta_{n,n+1}^{\ast}(Q)\leq 1.
\]

\noindent Hence by taking logarithms, there is a constant $C$ depending only on $n$ such that

\begin{equation} \label{eq:zokz}
 \left\vert \sum_{i=1}^{n+1} \nu_{n,i}^{\ast}(Q)\right\vert \leq \frac{C(n)}{\log(Q)}.
\end{equation}

\noindent The constants $w_{n}^{\ast}(\zeta),\widehat{w}_{n}^{\ast}(\zeta)$ are closely linked to the constants $w_{n,j}(\zeta),\widehat{w}_{n,j}(\zeta)$, 
as already indicated above.
For instance we have $\frac{w_{n}(\zeta)+1}{2}\leq w_{n}^{\ast}(\zeta)\leq w_{n}(\zeta)$, which by (\ref{eq:hundk}) implies the bound                            
(\ref{eq:anelim}) for all $\zeta$ mentioned in subsection 1.1. See also remark 3 to Corollary \ref{korol} in section 2.2.                                                                           
 In fact, we will prove

\begin{theo}  \label{thm1}
 For any integer $n\geq 1$ and any $\zeta\in{\mathbb{R}}$ not algebraic of degree $\leq n$ we have

\begin{eqnarray}
 w_{n}^{\ast}(\zeta)&\geq& w_{n,n+1}(\zeta)                         \label{eq:gans}        \\    
 \widehat{w}_{n}^{\ast}(\zeta)&\geq& \widehat{w}_{n,n+1}(\zeta).   \label{eq:joerh}
\end{eqnarray}

\end{theo}

\noindent Note that by (\ref{eq:4}),(\ref{eq:ka2}),(\ref{eq:ka3}) one can directly compute $w_{n,n+2-j}^{\prime}(\zeta)$ from $\widehat{w}_{n,j}(\zeta)$ 
as well as $\widehat{w}_{n,n+2-j}^{\prime}(\zeta)$ from $w_{n,j}(\zeta)$ for $1\leq j\leq n+1$, which leads to

\begin{equation} \label{eq:kb}
 w_{n,n+2-j}^{\prime}(\zeta)= \frac{1}{\widehat{w}_{n,j}(\zeta)}, \qquad  \widehat{w}_{n,n+2-j}^{\prime}(\zeta)= \frac{1}{w_{n,j}(\zeta)}.
\end{equation}

\noindent Combining (\ref{eq:kb}) with Theorem \ref{thm1} we immediately obtain

\begin{theo} \label{thm2}
 For any integer $n\geq 1$ and any $\zeta\in{\mathbb{R}}$ not algebraic of degree $\leq n$ we have

\begin{eqnarray*}
 w_{n}^{\ast}(\zeta)&\geq& \frac{1}{\widehat{w}_{n}^{\prime}(\zeta)}    \\    
 \widehat{w}_{n}^{\ast}(\zeta)&\geq& \frac{1}{w_{n}^{\prime}(\zeta)}.
\end{eqnarray*}

\end{theo}

\noindent It remains to prove Theorem \ref{thm1}. A basic observation linking $w_{n}^{\ast}(\zeta)$ with the constants $w_{n,j}(\zeta)$
 is the fact that any non-zero polynomial $P\in{\mathbb{Z}[T]}$ with $P^{\prime}(\zeta)\neq 0$ has a root $\alpha$ satisfying 

\begin{equation}  \label{eq:acc}
 \vert \zeta-\alpha\vert \leq n\left\vert \frac{P(\zeta)}{P^{\prime}(\zeta)}\right\vert,
\end{equation}

\noindent since for $P(\zeta)= \prod_{1\leq i\leq n}(\zeta-\alpha_{i})$ with $\alpha$ the closest zero to $\zeta$ (i.e. minimizing $\vert \zeta-\alpha\vert$) we have 
$\left\vert\frac{P^{\prime}(\zeta)}{P(\zeta)}\right\vert= \vert\sum_{1\leq i\leq n} \frac{1}{\zeta-\alpha_{i}}\vert\leq n\max_{1\leq i\leq n}\frac{1}{\vert \zeta-\alpha_{i}\vert}=n\frac{1}{\vert \zeta-\alpha\vert}$.
So in order to get a sequence of good apprimation values $\alpha$ for a fixed $\zeta$, we only need to find a sequence of polynomials
 with small values $\left\vert\frac{P(\zeta)}{P^{\prime}(\zeta)}\right\vert$.
Note that the logarithm of $P(\zeta)$ to the basis $H(P)$ (see section 1.1) is directly connected to $w_{n}(\zeta),\widehat{w}_{n}(\zeta)$. Putting

\begin{equation} \label{eq:mader}
 w_{n}^{\ast}(\zeta,H):= \min_{P:H(P)\leq H} -\frac{\log \left\vert\frac{P(\zeta)}{P^{\prime}(\zeta)}\right\vert}{\log H}-1
\end{equation}

\noindent where for all fixed $H$, $P$ runs through all polynomials $P$ of height $H(P)\leq H$, by virtue of (\ref{eq:acc}) one easily deduces

\begin{eqnarray}
 w_{n}^{\ast}(\zeta)&\geq& \limsup_{H\to\infty}  w_{n}^{\ast}(\zeta,H)             \label{eq:hans}        \\                  
 \widehat{w}_{n}^{\ast}(\zeta)&\geq& \liminf_{H\to\infty} w_{n}^{\ast}(\zeta,H).     \label{eq:joerg}                       
\end{eqnarray}

\noindent 

\section{Proof of Theorem \ref{thm1}}

\subsection{Strategy of the proof}

In fixed dimension $n$ we consider two successive minima problems. On the one hand the successive minima $\eta_{n,j}^{\ast}(Q)$ of the 
bodies $K^{+}(Q)$ with respect to the lattice $\Lambda^{+}$ defined in (\ref{eq:ask}),(\ref{eq:askk}) and the resulting functions $\nu_{n,j}^{\ast}(Q)$ 
 arising from $\eta_{n,j}^{\ast}(Q)$ from section 1.2.                                                                                             
On the other hand, we compress the bodies $K^{+}(Q)$ in direction orthogonal to the hyperplane $P^{\prime}(\zeta)=0$ by a fixed factor and consider
 their successive minima with respect to $\Lambda^{+}$.
By applying Minkowski's Theorem (\ref{eq:mink}) to both systems we will infer that at least one successive minimum of these two systems must 
differ from the corresponding successive minimum of the other system.
This will imply the existence of lattice points (in fact points corresponding to a successive minimum!) with 
relatively ''large'' values $\vert P^{\prime}(\zeta)\vert \approx H(P)$,
which is helpful for lower bounds for $w_{n}^{\ast}(\zeta),\widehat{w}_{n}^{\ast}(\zeta)$ in view of (\ref{eq:acc}) (or (\ref{eq:mader})). 
Assuming this point corresponds to the last (i.e. $(n+1)$-st) successive minimum gives a lower estimate for $w_{n}^{\ast}(\zeta),\widehat{w}_{n}^{\ast}(\zeta)$
 and together with the very intuitive geometric Lemma \ref{lemma} (although slightly techincal to prove) concerning the volume of the compressed bodies,
(\ref{eq:ka3}) and (\ref{eq:hans}) resp. (\ref{eq:joerg})
 leads to the estimates (\ref{eq:gans}) resp. (\ref{eq:joerh}) in Theorem \ref{thm1}.

\subsection{Exact proof of Theorem \ref{thm1}}

We will assume $\zeta$ fixed and identify a point $P=(x,y_{1},\ldots,y_{n})\in{\mathbb{Z}^{n+1}}$ with the polynomial $P(\zeta)=x+\zeta y+\ldots +\zeta^{n}y_{n}$.
For technical reasons we will call the successive minima problem concerning $K^{+}(Q)$ and $\Lambda^{+}$ {\em system A} throughout section 2.
We denote it with superscript $A$ and for simplicity we write $\eta_{n,j}^{A}(Q):= \eta_{n,j}^{\ast}(Q), \nu_{n,j}^{A}(Q):=\nu_{n,j}^{\ast}(Q)$
 with $\eta_{n,j}^{\ast}(Q),\nu_{n,j}^{\ast}(Q)$ as defined in section 1.2. \\
Furthermore, define {\em system B} as the successive minimum problem concerning $\Lambda^{+}$ and the convex body

\begin{equation}  \label{eq:fritzi}
 \mathscr{K}(Q):=K^{+}(Q)\cap c(Q)\cdot W(Q)
\end{equation}

\noindent with

\[
 W(Q):= \left\{(x,y_{1},\ldots ,y_{n})\in{\mathbb{R}^{n+1}}: \left\vert y_{1}+2\zeta y_{2}+\cdots+ n\zeta^{n-1}y_{n}\right\vert \leq Q^{\frac{1}{n}}\right\}
\]

\noindent and positive real numbers $c(Q)$ to be chosen later.                     
 Observe that $W(Q)$ is just the set of points $P\in{\mathbb{R}^{n+1}}$ with $\vert P^{\prime}(\zeta)\vert \leq Q^{\frac{1}{n}}$, so the convex
bodies $K^{+}(Q)$ are ''somehow compressed'' by some factor $c(Q)$ in the direction orthogonal to the hyperplane $P^{\prime}(\zeta)=0$ 
(although this is not quite true as the boundary changes shape).
Successive minima functions $\eta_{n,j}^{B}(Q), \nu_{n,j}^{B}(Q)$ arise from system B similarly as in system A. Note that by construction we have
$\mathscr{K}(Q)\subset K^{+}(Q)$ and therefore $\eta_{n,j}^{B}(Q)\leq \eta_{n,j}^{A}(Q)$ and $\nu_{n,j}^{B}(Q)\leq \nu_{n,j}^{A}(Q)$ for every $1\leq j\leq n+1$ and $Q>1$.  \\

\noindent We now choose the constants $c(Q)$ in (\ref{eq:fritzi}) such that

\begin{equation} \label{eq:nwesw}
 \rm{vol}(\mathscr{K}(Q))= \frac{1}{2(n+1)!}\rm{vol}(K^{+}(Q))<\frac{1}{(n+1)!}\rm{vol}(K^{+}(Q))=\frac{2^{n+1}}{(n+1)!}.
\end{equation}

\noindent Clearly, this is possible as the volume of $\mathscr{K}(Q)$ with arbitrary $c(Q)$ in (\ref{eq:fritzi}) 
depends continuously on $c(Q)$ and for sufficiently large $c(Q)\geq c_{0}(Q)$
we have $K^{+}(Q)=\mathscr{K}(Q)$, so in particular $\rm{vol}(K^{+}(Q))=\rm{vol}(\mathscr{K}(Q))$, as well as $\rm{vol}(\mathscr{K}(Q))=0$ for $c(Q)=0$.
By the intermediate value theorem and as the volume increases strictly as $c(Q)$ increases (as long as $\mathscr{K}(Q)\subsetneq K^{+}(Q)$), 
there is a unique $c(Q)$ with (\ref{eq:nwesw}) for every $Q>1$.\\
By (\ref{eq:mink}) we infer that for any $Q>1$, for at least one $k=k(Q)\in{\{1,2,\ldots,n+1\}}$ we have strict inequality $\eta_{n,k}^{B}(Q)< \eta_{n,k}^{A}(Q)$.    
By the definition of successive minima and the choice of our convex bodies $\mathscr{K}(Q)$ and $K^{+}(Q)$ this gives the existence of 
vectors $\boldsymbol{d}(Q)\in{\mathbb{Z}^{n+1}}$ with

\[
\boldsymbol{d}(Q)\in{\left( Q^{\nu_{n,k}^{A}(Q)}K^{+}(Q) \right) \setminus{\left( Q^{\nu_{n,k}^{B}(Q)}{\mathscr{K}(Q)} \right)}}. 
\]

\noindent We will again identify any such $\boldsymbol{d}(Q)=(x,y_{1},\ldots,y_{n})$ with the corresponding polynomial $P(\zeta)=x+\zeta y_{1}+\ldots +\zeta^{n} y_{n}$.
We consider these polynomials $P(\zeta)$ as $Q$ increases and will drop the dependence of $P$ from $Q$ in the notation as no misunderstandings can occur.  
 Since $\mathscr{K}(Q)$ only differs from $K^{+}(Q)$ in direction orthogonal to $P^{\prime}(\zeta)=0$ we have

\begin{equation} \label{eq:propellar}
\vert P^{\prime}(\zeta)\vert> Q^{\nu_{n,k}^{A}(Q)}\cdot c(Q)Q^{\frac{1}{n}}= c(Q)Q^{\nu_{n,k}^{A}(Q)+\frac{1}{n}}.
\end{equation}
  
\noindent On the other hand, as $\boldsymbol{d}(Q)\in{Q^{\nu_{n,j}^{A}(Q)}}K^{+}(Q)$ we have  

\begin{eqnarray}
 \vert P(\zeta)\vert &\leq& Q^{-1+\nu_{n,k}^{A}(Q)},   \label{eq:prepellar}  \\
 H(P) &\ll& Q^{\frac{1}{n}+\nu_{n,k}^{A}(Q)}.                 \label{eq:prapellar}
\end{eqnarray}

\noindent with constants in $\ll$ depending only on $n,\zeta$. More precisely, as by defintion we have $\vert y_{t}\vert \leq Q^{\frac{1}{n}+\nu_{n,k}^{A}(Q)}$ for $1\leq t\leq k$
and clearly $\vert P(\zeta)\vert \leq 1$ for $Q$ sufficiently large, we have the estimates

\begin{equation}  \label{eq:stojko}
\vert x\vert \leq \vert P(\zeta)\vert+(1+\vert \zeta\vert +\cdots +\vert \zeta\vert^{n})\max_{1\leq t\leq n} \vert y_{t}\vert\leq 1+(1+\vert \zeta\vert +\cdots +\vert \zeta\vert^{n}) Q^{\frac{1}{n}+\nu_{n,k}^{A}(Q)}. 
\end{equation}

\noindent So we infer $H(P)\leq 1+(1+\vert \zeta\vert +\cdots +\vert \zeta\vert^{n}) Q^{\frac{1}{n}+\nu_{n,k}^{A}(Q)}$ and thus (\ref{eq:prapellar}).\\
Note that (\ref{eq:propellar}),(\ref{eq:prepellar}),(\ref{eq:prapellar}) hold for any large $Q$ and $\nu_{n,k}^{A}(Q)\leq \nu_{n,n+1}^{A}(Q)$ as well as
$\underline{\nu}_{n,n+1}^{A}-\epsilon\leq \nu_{n,n+1}(Q)\leq \overline{\nu}_{n,n+1}+\epsilon$ for all $\epsilon>0$ and $Q\geq Q(\epsilon)$. So on the one hand 
we can choose a sequence of values $(Q_{s})_{s\geq 1}\to\infty$ such that the corresponding polynomials for any $\epsilon>0$ 
and sufficiently large $Q\geq Q_{0}(\epsilon)$ satisfy

\begin{eqnarray}
 \vert P(\zeta)\vert &\leq& Q^{-1+\underline{\nu}_{n,n+1}^{A}+\epsilon}  \label{eq:nadara}  \\
 \vert P^{\prime}(\zeta)\vert &\geq& c(Q)Q^{\frac{1}{n}+\underline{\nu}_{n,n+1}^{A}}  \label{eq:nedere} \\
 H(P) &\ll& Q^{\frac{1}{n}+\underline{\nu}_{n,n+1}^{A}}   \label{eq:nuduru}
\end{eqnarray}

\noindent with constants depending only on $n,\zeta$ in $\ll$. \\
On the other hand, for {\em any} sufficiently large $Q\geq Q_{0}(\epsilon)$ we clearly have

\begin{eqnarray}
 \vert P(\zeta)\vert &\leq& Q^{-1+\overline{\nu}_{n,n+1}^{A}+\epsilon}  \label{eq:nadaraz}  \\
 \vert P^{\prime}(\zeta)\vert &\geq& c(Q)Q^{\frac{1}{n}+\overline{\nu}_{n,n+1}^{A}}  \label{eq:nederez} \\
 H(P) &\ll& Q^{\frac{1}{n}+\overline{\nu}_{n,n+1}^{A}}.   \label{eq:nuduruz}
\end{eqnarray}

\noindent Assume in our present situation, i.e. $c(Q)$ defined by (\ref{eq:fritzi}),(\ref{eq:nwesw}), we already knew 

\begin{equation}  \label{eq:mileniza}
 \liminf_{Q\to\infty}\log_{Q}c(Q)\geq 0,
\end{equation}

\noindent which will be shown in Corollary \ref{korol} from Lemma \ref{lemma}. Then a choice of polynomials leading to
 (\ref{eq:nadara}),(\ref{eq:nedere}),(\ref{eq:nuduru}) gives in combination 
with (\ref{eq:mader}),(\ref{eq:hans}) and $\epsilon\to 0$

\begin{equation} \label{eq:normmal}
 w_{n}^{\ast}(\zeta)+1\geq -\frac{\log\left(\frac{P(\zeta)}{P^{\prime}(\zeta)}\right)}{\log(H(P))}\geq \frac{n+1}{n}\frac{1}{\frac{1}{n}+\underline{\nu}_{n,n+1}^{A}}= w_{n,n+1}(\zeta)+1,
\end{equation}

\noindent the equality on the right being just a variaton of (\ref{eq:ka3}).
Similarly (\ref{eq:nadaraz}),(\ref{eq:nederez}),(\ref{eq:nuduruz}) gives in combination with (\ref{eq:mader}),(\ref{eq:joerg}),(\ref{eq:ka3}) and $\epsilon\to 0$

\begin{equation} \label{eq:normmalz}
 \widehat{w}_{n}^{\ast}(\zeta)+1\geq -\frac{\log\left(\frac{P(\zeta)}{P^{\prime}(\zeta)}\right)}{\log(H(P))}\geq \frac{n+1}{n}\frac{1}{\frac{1}{n}+\overline{\nu}_{n,n+1}^{A}}= \widehat{w}_{n,n+1}(\zeta)+1.
\end{equation}

\noindent Subtracting one from both sides of (\ref{eq:normmal}),(\ref{eq:normmalz}) establishes the assertions of Theorem \ref{thm1}.\\ 
It remains to prove (\ref{eq:mileniza}). In fact, we prove that $c(Q)$ is even bounded below uniformly in the parameter $Q$.

\begin{lemm} \label{lemma}[Geometric lemma]\\
 Given $n\geq 2,\zeta\in{\mathbb{R}}$ as well as positive real parameters $R$ and $Q>1$, consider the sets 

\begin{eqnarray*}
 \chi_{A}(Q)&:=& \left\{(x,y_{1},\ldots ,y_{n})\in{\mathbb{R}^{n+1}}: \vert P(\zeta)\vert\leq Q^{-1}\right\}    \\
 \chi_{B}(R)&:=& \left\{(x,y_{1},\ldots ,y_{n})\in{\mathbb{R}^{n+1}}: \vert P^{\prime}(\zeta)\vert\leq R\right\}  \\
 \chi_{C}(Q)&:=& \left\{(x,y_{1},\ldots ,y_{n})\in{\mathbb{R}^{n+1}}: \vert y_{t}\vert \leq Q^{\frac{1}{n}}, \quad 1\leq t\leq n\right\},
\end{eqnarray*}

\noindent where $P(\zeta)=x+\zeta y_{1}+\cdots +\zeta^{n}y_{n}$.\\
Then for sufficiently large $Q$ and all $R$ we have

\[
 \rm{vol}\left(\chi_{A}(Q)\cap \chi_{B}(R)\cap \chi_{C}(Q)\right)\leq ERQ^{-\frac{1}{n}}
\]

\noindent with some constant $E=E(n,\zeta)$ independent of $Q$. 

\end{lemm}

\begin{proof}
Note first that $\chi_{A}(Q)$ is bounded by the two translates of the {\em fixed} hyperplane
$P(\zeta)=0$ by $Q^{-1}$ in direction orthogonal to $P(\zeta)=0$ to both sides of this hyperplane $P(\zeta)=0$.
 In particular $\chi_{A}(Q)$ converges to $P(\zeta)=0$ for $Q\to\infty$, which is not so important, however.
 Similarly, $\chi_{B}(R)$ is the space between two hyperplanes parallel to the hyperplane $P^{\prime}(\zeta)=0$ with distance $R$ in both directions
from the hyperplane $P^{\prime}(\zeta)=0$. 
Note that with respect to any other {\em fixed} direction $v\in{\mathbb{R}^{n+1}}$ with $v\notin{\{P:P(\zeta)=0\}}$ resp. $v\notin{\{P:P^{\prime}(\zeta)=0\}}$
(which is equivalent to $P(v)\neq 0$ resp. $P^{\prime}(v)\neq 0$)
 $\chi_{A}(Q)$ resp. $\chi_{B}(R)$ has width at most $N_{1}Q^{-1}$ resp. $N_{2}R$ for fixed constants $N_{1},N_{2}$ depending on $v$ but
 independent from $Q,R$. Throughout the proof we will make use of this for some given $v$ determined by $n,\zeta$.\\
Finally, $\chi_{C}(Q)$ just bounds the coordinates $(y_{1},y_{2},\ldots, y_{n})$ of
the vectors $(x,y_{1},\ldots ,y_{n})\in{\mathbb{R}^{n+1}}$ in dependence of $Q$.   \\

\noindent Observe that by these restrictions for $\chi_{C}(Q)$ and in view of the left hand inequality in (\ref{eq:stojko})
 we can assume $\vert x\vert \leq C(n,\zeta)Q^{\frac{1}{n}}$ 
for some constant $C(n,\zeta)$ independent of $Q$ without loss of generality.
 Thus every coordinate $(x,y_{1},\ldots,y_{n})$ of a point in   
$\chi_{A}(Q)\cap \chi_{B}(R)\cap \chi_{C}(Q)$ is bounded by $C(n,\zeta)Q^{\frac{1}{n}}$. 
Consequently for every $Q>1$ any rotation of the set $\chi_{A}(Q)\cap \chi_{B}(R)\cap \chi_{C}(Q)$
lies in the centralsymmetric                                                                                              
cube with side length $\sqrt{n}C(n,\zeta)Q^{\frac{1}{n}}$ and surfaces parallel to the hyperplanes $x=0,y_{1}=0,\ldots,y_{n}=0$ given by

\[
\kappa(Q):= \left\{(x,y_{1},\ldots,y_{n}):  \vert x\vert \leq \sqrt{n}C(n,\zeta)Q^{\frac{1}{n}},
\vert y_{t}\vert \leq \sqrt{n}C(n,\zeta)Q^{\frac{1}{n}}, 1\leq t\leq k\right\}.
\]
 
\noindent We apply a rotation $\varpi$ on $\mathbb{R}^{n+1}$ in such a way that the hyperplane $P(\zeta)=0$ is sent to the hyperplane 
$H_{0}$ defined by $x=0$. 
 Let $H_{1}$ be the image of the hyperplane $P^{\prime}(\zeta)=0$ under $\varpi$. Clearly $H_{0}\neq H_{1}$
as $P^{\prime}(\zeta)$ has lower degree than $P(\zeta)$ and hence the intersection $H_{0}\cap H_{1}$ has dimension $(n-1)$.
As rotations preserve volumes,

\[
 \rm{vol}\left(\varpi(\chi_{A}(Q)\cap \chi_{B}(R)\cap \chi_{C}(Q))\right)= \rm{vol}(\chi_{A}(Q)\cap \chi_{B}(R)\cap \chi_{C}(Q)).
\]

\noindent Moreover note also that intersecting $\chi_{A}(Q)\cap \chi_{B}(R)\cap \chi_{C}(Q)$
 with $\kappa(Q)$ doesn't change the volume as stated (we may replace $\chi_{C}(Q)$ by $\kappa(Q)$).\\
More generally, for real numbers $a,b$ define 

\begin{eqnarray*}
H_{0,a}&:=&\varpi\left(\left\{P:P(\zeta)=a\right\}\right)=\left\{(x,y_{1},\ldots ,y_{n})\in{\mathbb{R}^{n+1}}:x=a\right\}, \\
H_{1,b}&:=&\varpi\left(\left\{P:P^{\prime}(\zeta)=b\right\}\right),
\end{eqnarray*}

\noindent such that in particular $H_{0}=H_{0,0}, H_{1}=H_{1,0}$.\\
\noindent From the preliminary descriptions of $\chi_{A}(Q),\chi_{B}(R)$ we easily see

\begin{eqnarray*}
 \varpi(\chi_{A}(Q))&=& \bigcup_{a\in{[-Q^{-1},Q^{-1}]}} H_{0,a}, \\
 \varpi(\chi_{B}(Q))&=& \bigcup_{b\in{[-R,R]}} H_{1,b},
\end{eqnarray*}

\noindent and by construction of $\kappa(Q)$ we conclude

\begin{equation} \label{eq:leniza}
 \varpi \left(\chi_{A}(Q) \cap \chi_{B}(R) \cap \chi_{C}(Q) \right) \subset \left( \cup_{a\in{[-Q^{-1},Q^{-1}]}} H_{0,a}\right) \bigcap \left( \cup_{b\in{[-R,R]}} H_{1,b} \right) \bigcap \kappa(Q).
\end{equation}

\noindent In view of (\ref{eq:leniza}) and since rotations don't change the volume 
it is sufficient to prove the upper estimate $ERQ^{-\frac{1}{n}}$
 for the volume of the right hand side of (\ref{eq:leniza}), i.e. 

\begin{equation}  \label{eq:milanas}
 \rm{vol}\left(\left(\cup_{a\in{[-Q^{-1},Q^{-1}]}} H_{0,a}\right)\bigcap \left(\cup_{b\in{[-R,R]}} H_{1,b}\right) \bigcap \kappa(Q)\right) \leq ERQ^{-\frac{1}{n}},
\end{equation}

\noindent to establish the assertions of the Lemma.\\
In order to do this we use Fubini's Theorem twice. We first give upper bounds for the $(n-1)$-dimensional
volumes of the intersections $H_{0,a}\cap H_{1,b}\cap \kappa(Q)$, then apply Fubini' Theorem to derive upper bounds
for the $n$-dimensional volumes of $H_{0,a}\cap \bigcup_{b\in{[-R,R]}} H_{1,b}\cap \kappa(Q)$ for every $a\in{[-Q^{-1},Q^{-1}]}$
and then again apply Fubini's Theorem by integrating these $n$-dimensional volumes along the $x$-axis to finally 
derive the required upper bound.\\  
Clearly, the $(n-1)$-dimensional volume of $H_{0}\cap H_{1}\cap \kappa(Q)$ is proportional to $Q^{\frac{n-1}{n}}$, let's say 
$\rm{vol}(H_{0}\cap H_{1}\cap \kappa(Q))= DQ^{\frac{n-1}{n}}$ for a constant $D$ depending only on
the angle between $H_{0}$ and $H_{1}$ which is determined by $n,\zeta$ (in particular independent of $Q$). 
Similarly we see that we can find a constant $D_{0}=D_{0}(n,\zeta)$ such that simultaneously for all $a,b\in{\mathbb{R}}$ we have 

\begin{equation} \label{eq:lovemilena}
 \rm{vol} (H_{0,a}\cap H_{1,b}\cap \kappa(Q)) \leq D_{0}Q^{\frac{n-1}{n}}, \qquad a,b\in{\mathbb{R}},
\end{equation}

\noindent as all $H_{0,a}\cap H_{1,b}$ are $(n-1)$-dimensional subspaces in a $n$-dimensional cube with side length proportional to $Q^{\frac{1}{n}}$.\\  
Since $\chi_{B}(R)$ (and so $\varpi(\chi_{B}(R))=\cup_{b\in{[-R,R]}} H_{1,b}$ too) has width $R$ in direction orthogonal to $H_{1,b}$, 
and in view of (\ref{eq:lovemilena}), we infer that 

\begin{equation} \label{eq:lumumbart}
 \rm{vol}_{n}\left(H_{0,a}\bigcap \left(\cup_{b\in{[-R,R]}} H_{1,b}\right) \bigcap \kappa(Q)\right) \leq D_{0}Q^{\frac{n-1}{n}}\cdot D_{1}R, \qquad a\in{\mathbb{R}}
\end{equation}
 
\noindent with some constant $D_{1}$ depending only on the angle between the $x$-axis and $H_{1}$, which itself is determined by $n,\zeta$.
 So we have estimated the $n$-dimensional volume of $\varpi\left(\chi_{B}(R)\cap \chi_{C}(Q)\right)=\left(\cup_{b\in{[-R,R]}} H_{1,b}\right) \cap \kappa(Q)$ 
in every hyperplane $H_{0,a}$. 
Observe that $\varpi\left(\chi_{A}(Q)\cap \chi_{B}(R)\cap \chi_{C}(Q)\right)$ has width $mQ^{-1}$ in direction of the $x$-axis
with a constant $m$ depending only on the angle between the hyperplane $P(\zeta)=0$ and the hyperplane $H_{0}$, which is again determined by $n,\zeta$ 
(so in particular independent of $Q$).
We apply Fubini's Theorem to (\ref{eq:lumumbart}) and conclude

\begin{equation*}
 \rm{vol}\left(\left(\cup_{a\in{[-Q^{-1},Q^{-1}]}} H_{0,a}\right)\bigcap \left(\cup_{b\in{[-R,R]}} H_{1,b}\right) \bigcap \kappa(Q)\right)\leq mQ^{-1}\cdot D_{0}D_{1}Q^{\frac{n-1}{n}}R=mD_{0}D_{1}Q^{-\frac{1}{n}}R.  
\end{equation*}

\noindent Inequality (\ref{eq:milanas}) and hence the assertion of the Lemma follows with $E=E(n,\zeta):=mD_{0}D_{1}$.

\end{proof}

\begin{coro}  \label{korol}
 In the context of the first part of the proof of Theorem \ref{thm1} (ie (\ref{eq:nwesw}) holds) we have $c(Q)\geq B$ for some constant $B$
uniformly in $Q$, in particular (\ref{eq:mileniza}) holds.
\end{coro}

\begin{proof}
 By definition of $c(Q)$ we have $c(Q)=R(Q)Q^{-\frac{1}{n}}$ for $R=R(Q)$ that satisfies 

\[
 \rm{vol}\left(\chi_{A}(Q)\cap \chi_{B}(R)\cap \chi_{C}(Q)\right)=\frac{2^{n+1}}{2(n+1)!}.
\]

\noindent We may apply Lemma \ref{lemma} and with respect to the constant $E=E(n,\zeta)$ of this lemma this yields

\[
 \frac{2^{n+1}}{2(n+1)!}\leq ER(Q)Q^{-\frac{1}{n}},
\]

\noindent or equivalently $B\cdot Q^{\frac{1}{n}}\leq R(Q)=Q^{\frac{1}{n}}c(Q)$ with $B:=\frac{2^{n+1}}{2(n+1)!E}$ for all $Q>1$  
and we conclude $c(Q)\geq B$ uniformly in the parameter $Q$.

\end{proof}

\noindent Thus we have established Theorem \ref{thm1} and consequently Theorem \ref{thm2}.\\

\noindent Remarks: 1) One can show $\rm{vol}\left(\chi_{A}(Q)\cap \chi_{B}(R)\cap \chi_{C}(Q)\right)\geq FRQ^{-\frac{1}{n}}$ 
for some constant $F=F(n,\zeta)$ independent from $Q$ with a proof similar to the one of Lemma \ref{lemma}.
 So by arguments very similar to those in the proof of Corollary \ref{korol}
 we have that in fact $c(Q)$ is also uniformly bounded above by a positive constant and consequently in combination with 
 (\ref{eq:mileniza}) we actually have $\lim_{Q\to\infty} \log_{Q}c(Q)=0$.\\
2) Observe that all the constants occuring throughout the proof of Lemma \ref{lemma} can be estimated explicitely
in dependence of $n,\zeta$ so we can write $c(Q)\leq K(n,\zeta)$ with an effective constant $K(n,\zeta)$ in (\ref{eq:nedere}),(\ref{eq:nederez}).
However, the $\epsilon$-term in the exponent of (\ref{eq:nadara}),(\ref{eq:nadaraz}) doesn't allow any improvements in 
(\ref{eq:normmal}),(\ref{eq:normmalz}) for any given $n,\zeta$. \\
3) Lemma 15 Chap. 3 of $\S 3$ in \cite{8} states that for a polynomial $P$ of degree $D$ and length $L$ and                                        
with zero $\alpha$ satisfying $\vert \zeta-\alpha\vert\leq 1$

\[
 \vert P(\zeta)\vert \leq \vert \zeta- \alpha\vert \cdot LD(1+\vert \zeta\vert)^{D-1}
\]

\noindent holds, leading to the well known result                                                                 

\[
 w_{n}^{\ast}(\zeta)\leq w_{n}(\zeta)=w_{n,1}(\zeta).
\]

\noindent So $w_{n}^{\ast}(\zeta)$ can be bounded above in terms of the approximation constants $w_{n,j}(\zeta)$ which
reverses the direction of the estimates in Theorem \ref{thm1}.

\subsection{Estimates for $w_{n}^{\ast}(\zeta)$ depending on $n$ only} 

A famous result of Wirsing \cite{10}                                                                                                                                    
states 

\begin{equation}  \label{eq:mily}
 w_{n}^{\ast}(\zeta) \geq \frac{w_{n}(\zeta)+1}{2}.
\end{equation}

\noindent We combine the results of Theorem \ref{thm1}, \ref{thm2} with (\ref{eq:mily}) to give lower bounds for $w_{n}^{\ast}(\zeta)$ uniformly in $\zeta$
not algebraic of degree $\leq n$. In order to do this, we use the funcions $\psi_{n,j}^{\ast}$ which have the useful property

\begin{equation} \label{eq:myli}
  \left\vert \sum_{j=1}^{n+1} \psi_{n,j}^{\ast}(Q)\right\vert \leq \frac{C(n)}{\log(Q)},
\end{equation}

\noindent which follows from (\ref{eq:zokz}) and (\ref{eq:mueli}), and $\psi_{n,j}^{\ast}$ relate to the constants $w_{n,j}(\zeta)$.
To get a slightly better result we will also use the dual version of (\ref{eq:stoiko}): indeed from  (\ref{eq:stoiko}) and (\ref{eq:ka2}) it follows that
for $\zeta$ not algebraic of degree $\leq n$ we have

\begin{equation}  \label{eq:stiko}
 \underline{\psi}_{n,j+1}^{\ast}\leq \overline{\psi}_{n,j}^{\ast}, \qquad 1\leq j\leq n.
\end{equation}

\noindent The following preliminary Proposition is a very easy consequence of (\ref{eq:mueli}) and (\ref{eq:stiko}).

\begin{prop} \label{mylli}
 Let $n\geq 2$ be an integer. Then for $\zeta$ not algebraic of degree $\leq n$ the relation 

\[
 \underline{\psi}_{n,1}^{\ast}\leq -\frac{2}{n-1} \underline{\psi}_{n,n+1}^{\ast}
\]

\noindent holds for the approximation constants $\psi_{n,1}^{\ast}, \psi_{n,n+1}^{\ast}$ associated to $(\zeta,\zeta^{2},\ldots,\zeta^{n})$. 
\end{prop}

\begin{proof}
 By definition for any $\epsilon>0$ and sufficiently large $Q$ we have $\psi_{n,n+1}^{\ast}(Q)\geq \underline{\psi}_{n,n+1}^{\ast}-\epsilon$.
 By (\ref{eq:stiko}) we also have $\underline{\psi}_{n,n+1}^{\ast}\leq \overline{\psi}_{n,n}^{\ast}$, so that
there exist arbitrarily large values of $Q$ such that $\psi_{n,n}^{\ast}(Q)\geq \psi_{n,n+1}^{\ast}(Q)-\epsilon$.
Combining these observations yields arbitrarily large values $Q$, 
such that $\psi_{n,n}^{\ast}(Q)+\psi_{n,n+1}^{\ast}(Q)\geq 2\underline{\psi}_{n,n+1}^{\ast}(Q)-2\epsilon$.\\
On the other hand, for any $Q$ we have $(n-1)\psi_{n,1}^{\ast}(Q)\leq \sum_{j=1}^{n-1} \psi_{n,j}^{\ast}(Q)$, in particular 
for those values $Q$ with the property $\psi_{n,n}^{\ast}(Q)+\psi_{n,n+1}^{\ast}(Q)\geq 2\underline{\psi}_{n,n+1}^{\ast}(Q)-2\epsilon$.
The assertion of the Proposition follows with (\ref{eq:myli}) und $\epsilon\to 0$.
\end{proof}

\noindent An application of Proposition \ref{mylli} together with Wirsing's result yields 

\begin{coro}  \label{corall}
 Let $n\geq 2$ be an integer. Then for $\zeta$ not algebraic of degree $\leq n$ the 
approximation constant $w_{n}^{\ast}(\zeta)$ of $(\zeta,\zeta^{2},\ldots,\zeta^{n})$ is bounded below as follows

\begin{equation} \label{eq:holnmm}
 w_{n}^{\ast}(\zeta)\geq \frac{1}{4}\left(n+1+\sqrt{n^{2}+10n-7}\right)=:\mathscr{U}(n). 
\end{equation}

\noindent For $n\to\infty$ we have the asymptotical behaviour $\mathscr{U}(n)=\frac{n}{2}+\frac{3}{2}+o(1)$.
\end{coro}

\begin{proof}
From (\ref{eq:ka3}) we have

\begin{equation}  \label{eq:ellenho}
 \underline{\psi}_{n,n+1}^{\ast}= \frac{n-w_{n,n+1}(\zeta)}{n(w_{n,n+1}(\zeta)+1)}.
\end{equation}

\noindent Dividing the right hand side of (\ref{eq:ellenho}) by $(n-1)$ and combining it with 
the estimate from Proposition \ref{mylli} gives an upper bound for
$\underline{\psi}_{n,1}^{\ast}$ in terms of $w_{n,n+1}(\zeta)$. Using this expression in (\ref{eq:ka3})
and applying Wirsing's result (\ref{eq:mily}) leads to 

\[
 w_{n}^{\ast}(\zeta)\geq \frac{w_{n}(\zeta)+1}{2}\geq \frac{n+1}{2}\cdot \frac{1}{1-\frac{2}{n-1}\frac{n-w_{n,n+1}(\zeta)}{w_{n,n+1}(\zeta)+1}}.
\]

\noindent On the other hand we have the lower bound $w_{n,n+1}(\zeta)$ for $w_{n}^{\ast}(\zeta)$ by Theorem \ref{thm1}, so

\[
 w_{n}^{\ast}(\zeta)\geq \max\left\{ \frac{n+1}{2}\cdot \frac{1}{1-\frac{2}{n-1}\frac{n-w_{n,n+1}(\zeta)}{w_{n,n+1}(\zeta)+1}}, w_{n,n+1}(\zeta)\right\}.
\]

\noindent It's not hard to see that the left hand term in the maximum decreases as $w_{n,n+1}(\zeta)$ increases, so the minimum is attained
at the value $w_{n,n+1}(\zeta)>0$ where both expressions in the maximum coincide. This leads to a quadratic equation and after basic simplifications
finally yields $\mathscr{U}(n)$ as the minimum lower bound for $w_{n}^{\ast}(\zeta)$. Checking the asymptotics for $\mathscr{U}(n)$ is a standard calculation.
\end{proof}

\noindent Remarks: 1) Lemma 1 on page 46 in \cite{5}                                                                                        
 states that in the present case of simultaneous approximation of $(\zeta,\zeta^{2},\ldots,\zeta^{n})$ of $\zeta$
not algebraic of degree $\leq \left\lceil \frac{n}{2}\right\rceil$
the approximation constant $\widehat{w}_{n}^{\prime}(\zeta)$ is bounded above by $\left \lceil \frac{n}{2}\right\rceil^{-1}$.
Applying this to (\ref{eq:kb}) and Theorem \ref{thm1} we immediately derive 

\[
 w_{n}^{\ast}(\zeta)\geq \left \lceil \frac{n}{2}\right\rceil
\]

\noindent and slight refinements in combination with Wirsing's result can be derived similarly as in Corollary \ref{corall}. 
However, the results of Corollary \ref{corall} are a little stronger.\\

\noindent 2) Note that there exists no nontrivial upper bound for the value $w_{n}^{\prime}(\zeta)$ (as for
 $\widehat{w}_{n}^{\prime}(\zeta)$ in Remark 1) even in the present special case 
of simultaneous approximation of $(\zeta,\zeta^{2},\ldots,\zeta^{n})$.
Indeed, we can have $w_{n}^{\prime}(\zeta)=\infty$ (which is equivalent to $\widehat{w}_{n,n+1}(\zeta)=0$ by (\ref{eq:kb})), 
taking $\zeta=\sum_{l\geq 1} 10^{-l!}$ for example.
So we cannot use Theorems \ref{thm1},\ref{thm2} to give nontrivial bounds for the approximation constant $\widehat{w}_{n}^{\ast}(\zeta)$
 as easily as above.\\
Moreover, no analogue of (\ref{eq:mily}) for $\widehat{w}_{n}^{\ast}(\zeta)$ seems to be known. Bugeaud and Laurent established in Theorem 2.1 in \cite{1}                     
for $\zeta$ not algebraic of degree $\leq n$ the inequality

\[
 \widehat{w}_{n}^{\ast}(\zeta)\geq \frac{w_{n}(\zeta)}{w_{n}(\zeta)-n+1},
\]

\noindent which we can combine with Theorem \ref{thm1} to get

\begin{equation}  \label{eq:langebeine}
 \widehat{w}_{n}^{\ast}(\zeta)\geq \max\left\{\frac{w_{n}(\zeta)}{w_{n}(\zeta)-n+1},\widehat{w}_{n,n+1}(\zeta)\right\}.
\end{equation}

\noindent However, any number $\zeta$ with $w_{n}^{\prime}(\zeta)=\infty$ (for instance again $\zeta=\sum_{l\geq 1} 10^{-l!}$)
automatically yields $w_{n}(\zeta)=\infty$ (which follows easily from the definition of $w_{n}(\zeta),w_{n}^{\prime}(\zeta)$ 
or alternatively from Khinchins transference principle $w_{n}(\zeta)\geq (n-1)w_{n}^{\prime}(\zeta)+n-2$, see \cite{4}).                                     
For such $\zeta$, by (\ref{eq:kb}) we also have $\widehat{w}_{n,n+1}(\zeta)=0$ though, 
so in this case (\ref{eq:langebeine}) only leads to the very weak bound $\widehat{w}_{n}^{\ast}(\zeta)\geq 1$.


\newpage 


\begin{thebibliography}{99}
\bibitem{1} Y. Bugeaud, M. Laurent: \emph{Exponents of Diophantine approximation and Sturmian continued fractions}, Ann. Inst. Fourier (Grenoble) $\boldsymbol{55}$ (2005), no. 3, p. 773-804   \newline
\bibitem{2} H. Davenport, W.M. Schmidt: \emph{Approximation to real numbers by algebraic integers}, Acta Arith. 15 (1969), p. 393-416  \newline
\bibitem{3} P.M. Gruber, C.G. Lekkerkerker: \emph{Geometry of numbers}, North-Holland Verlag (1987)  \newline
\bibitem{4} Y.A. Khintchine [A. Ya. Khinchin]: \emph{Zur metrischen Theorie der diophantischen Approximationen}, Math. Z. 24 (1926), 706-714            \newline
\bibitem{5} M. Laurent: \emph{On simultaneous rational approximation to successive powers of a real number}, Indag. Math. (N.S.) 14 (2003), no. 1,p. 45-53 \newline
\bibitem{6} W.M. Schmidt, L. Summerer: \emph{Parametric geometry of numbers and applications}, Acta Arithm. 140.1 (2009)      \newline
\bibitem{7} W.M. Schmidt, L. Summerer: \emph{Diophantine approximation and parametric geometry of numbers}, Monatshefte f\"ur Mathematik Vol 169, Issue 1 (2013) p. 54-107                                    \newline
\bibitem{8} T. Schneider: \emph{Einf\"uhrung in die transzendenten Zahlen}, Springer-Verlag Berlin (1957)      \newline
\bibitem{9} M. Waldschmidt: \emph{Report on some recent advances in Diophantine approximation} (2009)                              \newline    
\bibitem{10} E. Wirsing: \emph{Approximation mit algebraischen Zahlen beschr\"ankten Grades},  J. Reine Angew. Math. 206 1960 67-77. \# 79  
\end{thebibliography}
\end{document}